\documentclass[12pt]{amsart}

\usepackage{amsmath, amssymb, amsthm, amscd}
\usepackage{url}
\usepackage{graphicx}
\usepackage[hidelinks,pagebackref,pdftex]{hyperref}
\usepackage{color}
\usepackage{import}
\usepackage[margin=1.2in]{geometry}
\usepackage{physics}
\usepackage[abs]{overpic}
\usepackage{contour}

\usepackage{marginnote}
\long\def\@savemarbox#1#2{\global\setbox#1\vtop{\hsize\marginparwidth 
  \@parboxrestore\tiny\raggedright #2}}
\renewcommand*{\backref}[1]{}
\renewcommand*{\backrefalt}[4]{
  \ifcase #1
  [No citations.]
  \or [#2]
  \else [#2]
  \fi }

\AtBeginDocument{%
   \def\MR#1{}
}

\numberwithin{equation}{section}
\theoremstyle{plain}
\newtheorem{theorem}[equation]{Theorem}

\newtheorem{lemma}[equation]{Lemma}

\newtheorem{corollary}[equation]{Corollary}

\newtheorem*{namedtheorem}{\theoremname}
\newcommand{\theoremname}{testing}

\theoremstyle{definition}

\newtheorem{example}[equation]{Example}

\newcommand{\from}{\colon} 
\newcommand{\h}{{\mathfrak{h}}}
\newcommand{\HH}{{\mathbb{H}}}
\newcommand{\RR}{{\mathbb{R}}}
\newcommand{\ZZ}{{\mathbb{Z}}}
\newcommand{\NN}{{\mathbb{N}}}
\newcommand{\CC}{{\mathbb{C}}}
\newcommand{\QQ}{{\mathbb{Q}}}

\newcommand{\Zw}{{\mathbb{Z}[\omega]}}
\newcommand{\PSLw}{{\mathrm{PSL}_2(\mathbb{Z}[\omega])}}
\newcommand{\SLw}{{\mathrm{SL}_2(\mathbb{Z}[\omega])}}
\newcommand{\PSLC}{{\mathrm{PSL}_2(\mathbb{C})}}
\newcommand{\SLC}{{\mathrm{SL}_2(\mathbb{C})}}

\newcommand{\Ii}{{\mathcal{I}}}
\newcommand{\Bb}{{\mathcal{B}}}
\newcommand{\Cc}{{\mathcal{C}}}
\newcommand{\Dd}{{\mathcal{D}}}
\newcommand{\Pp}{{\mathcal{P}}}

\newcommand{\Oo}{{\mathcal{O}}}
\newcommand{\Ll}{{\mathcal{L}}}

\newcommand{\Gm}{{\Gamma}}
\newcommand{\gm}{{\gamma}}
\newcommand{\mg}{{\omega}}
\newcommand{\kp}{{\kappa}}
\newcommand{\md}{{\lambda}}
\newcommand{\af}{{\alpha}}
\newcommand{\dt}{{\delta}}

\newcommand{\xe}{{\smqty(\xi\\ \eta)}}
\newcommand{\xef}{{\frac{\xi}{\eta}}}

\newcommand{\bdy}{\partial}

\newcommand{\To}{\longrightarrow}

\newcommand{\refthm}[1]{Theorem~\ref{Thm:#1}}

\newcommand{\refeqn}[1]{\eqref{Eqn:#1}}

\newcommand{\refsec}[1]{Section~\ref{Sec:#1}}
\newcommand{\reffig}[1]{Figure~\ref{Fig:#1}}

\title{Lambda Lengths In The Figure Eight Knot Complement}
\author{Joshua A. Howie} 
\address{School of Mathematics, Monash University, 9 Rainforest Walk, Clayton VIC 3800, Australia}
\email{josh.howie@monash.edu}

\author{Dionne Ibarra}
\address{School of Mathematics, Monash University, 9 Rainforest Walk, Clayton VIC 3800, Australia}
\email{dionne.ibarra@monash.edu}

\author{Daniel V. Mathews}
\address{School of Mathematics, Monash University, 9 Rainforest Walk, Clayton VIC 3800, Australia; School of Physical and Mathematical Sciences, Nanyang Technological University, 21 Nanyang Link, Singapore 637371}
\email{dan.v.mathews@gmail.com}

\author{Lecheng Su}
\address{School of Mathematics, Monash University, 9 Rainforest Walk, Clayton VIC 3800, Australia}
\email{lecheng.su@monash.edu}

\thanks{\today}

\begin{document}

\begin{abstract}
In the complete hyperbolic structure on the complement of the figure eight knot, we determine the set of lambda lengths from the maximal cusp to itself. Using the correspondence between spinors and spin-decorated horospheres, we show that these lambda lengths are precisely the Eisenstein integers, up to multiplication by a unit. We also show that the inter-cusp distances from the maximal cusp to itself are precisely the norms of Eisenstein integers.
\end{abstract}

\maketitle


\section{Introduction}

In \cite{mat23}, the third author introduced a correspondence between certain types of spinors --- the \emph{spin-vectors} of Penrose and Rindler \cite{Penrose_Rindler84} ---  and horospheres in hyperbolic 3-space $\HH^3$ with certain decorations. This construction, combining Penrose and Rindler's work with that of Penner \cite{Penner87}, has been applied to generalise Descartes' circle theorem \cite{MatZym23}, and extends to higher dimensions \cite{MatVar}. In this paper, we show that the spinor--horosphere correspondence can be used to calculate lambda lengths and inter-cusp distances in cusped hyperbolic 3-manifolds, using the example of the figure eight knot complement.

Throughout this paper, $K$ is the figure-eight knot in $S^3$, and $M = S^3 \setminus K$ is its complement. The unique complete hyperbolic structure on $M$ is given by a developing map $\mathfrak{D} \colon \widetilde{M} \To \HH^3$, unique up to isometry in $\HH^3$. There is a reasonably canonical choice of $\mathfrak{D}$ which appears in the literature, and we describe it below.

A cusp neighbourhood of $M$ appears via $\mathfrak{D}$ as an infinite family of horoballs in $\HH^3$. For a small enough cusp neighbourhood, these horoballs are disjoint; expanding the cusp neighbourhood until there is a tangency, we obtain \emph{maximal horoballs} corresponding to a \emph{maximal cusp neighbourhood} of $K$. We call the horospheres bounding these maximal horoballs the \emph{canonical horospheres} of $M$. Some pairs are tangent, by definition, but all other pairs are disjoint. A natural question arising is, what are the distances between these horospheres?

Given the spinor--horosphere correspondence, we can also ask, what are the spinors corresponding to the canonical horospheres of $M$? This correspondence requires a decoration on the horospheres, but there are natural choices, as we discuss below.

An \emph{intercusp geodesic} $\gamma$ is then the image in $M$ of a shortest geodesic in $\HH^3$ from one canonical horosphere to another. Since $M$ has only one cusp, and the geometry on the cusp is Euclidean, there is a well-defined notion of direction along the cusp, and we can ask: after parallel translation along $\gamma$, how are directions along the cusp affected? \emph{Lambda lengths} contain this information, and can be calculated straightforwardly from spinors, as discussed in \cite{mat23}.

In this paper, we answer all the above questions, proving the following three theorems.

\begin{theorem}
\label{Thm:Zw_horospheres}
If $\smqty(\xi\\ \eta)\in\CC^2$ is a spinor arising in the canonical hyperbolic structure on $S^3\setminus K$, then $\xi$ and $\eta$ are relatively prime Eisenstein integers. Moreover, given a pair of relatively prime Eisenstein integers $(\xi,\eta)$, there exists a unit Eisenstein integer $u$ such that $u\smqty(\xi\\ \eta)$ is a spinor arising in the canonical hyperbolic structure on $S^3\setminus K$.
\end{theorem}

\begin{theorem}
\label{Thm:lambda_lengths_horospheres}
The $\md$-length between any two spin-decorated horospheres in the canonical hyperbolic structure on $S^3\setminus K$ is an Eisenstein integer. Moreover, for any Eisenstein integer $\alpha$,
there exists a unit Eisenstein integer $u$ such that $u\alpha$ is the $\md$-length between two spin-decorated horospheres in the canonical hyperbolic structure on $S^3\setminus K$. 
\end{theorem}

\begin{theorem}
\label{Thm:horosphere_distances}
The set of hyperbolic distances between canonical horospheres in the complete hyperbolic structure on $S^3\setminus K$ is precisely the set $\{2\log{|\alpha|}\mid \alpha\in\Zw\setminus\{0\}\}$.
\end{theorem}

Note for $\alpha \in \Zw$ we have $2\log{|\alpha|}= \log{(|\alpha|^2)} = \log N(\alpha)$, where $N(\alpha) = |\alpha|^2$ is the number-theoretic norm of $\alpha$. The norms of Eisenstein integers are well known (e.g. \cite[sec. 9.1]{IreRos90}). Every prime in $\Zw$ has norm of one of the following forms: (i) $p$, where $p$ is a rational prime such that $p \equiv 1$ mod $3$; (ii) $p^2$, where $p$ is a rational prime such that $p \equiv 2$ mod $3$; (iii) $3$. So the final theorem has the following more explicit corollary.
\begin{corollary}
The set of hyperbolic distances between canonical horospheres in the complete hyperbolic structure on $S^3 \setminus K$ is precisely the set
\[\left\{ \log n \mid n = \prod_{p} p^{k_p}, \quad p \text{ a rational prime}, \quad k_p \in \NN\cup\{0\},\quad k_p \text{ even for } p \equiv 2 \text{ mod 3}. \right\}\]
\qed
\end{corollary}

Previous work has studied intercusp distances. For instance, in 2014 Thistlethwaite and Tsvietkova \cite{ThiTsv14} gave a method for calculating hyperbolic structures on link complements, using variables called \emph{edge labels} and \emph{crossing labels}. The latter, also known as \emph{intercusp parameters}, were shown by Neumann and Tsvietkova \cite{NeuTsv16} in 2016 to be closely related to the invariant trace field. In our language, an intercusp parameter is a lambda length raised to the power of $-2$.

In 2018 Kim, Kim and Yoon \cite{KKY18} studied octahedral decomposition of knot complements, and their hyperbolic geometry, using these parameters (there discussed as \emph{segment} and \emph{crossing} labels). In the 2023 sequel \cite{KKY23}, they studied the relationship of intercusp parameters (there described as \emph{long-edge parameters}) to Ptolemy coordinates and varieties.

It is interesting to compare these results with those of McShane~\cite{mcs24}, relating $\md$-lengths on the thrice-punctured sphere with Eisenstein integers.

\subsection*{Structure of this paper}
This paper is organised as follows. We first recall facts needed for the proofs: Eisenstein intgers in \refsec{Zw}; the figure eight knot complement in \refsec{fig8}; and spinors, in \refsec{spinors}. In \refsec{holonomy_action} we study the action of the holonomy group on horospheres via isometric circles, and use this in \refsec{Zw_horospheres} to prove \refthm{Zw_horospheres}, that relatively prime Eisenstein integer spinors correspond to the desired horospheres. In \refsec{lambda_lengths} we consider lambda lengths between spin-decorated horospheres, and prove \refthm{lambda_lengths_horospheres} and \refthm{horosphere_distances}. Finally, in \refsec{questions} we raise some directions for further research.

\subsection*{Acknowledgments}
This work grew out of a research collaboration at the MATRIX workshop ``Low Dimensional Topology: Invariants of Links, Homology Theories, and Complexity" in June 2024. The authors would like to acknowledge all those who participated in discussions over the course of the workshop: Grace Garden, Connie On Yu Hui, Adele Jackson, Emma McQuire, Jessica Purcell, Corbin Reid, and Em Thompson. We would also like to acknowledge those who have participated in further discussions at Monash University on related topics, including Nicola Harif and Orion Zymaris. This work was supported by ARC grant DP210103136.


\section{Eisenstein Integers}
\label{Sec:Zw}

The following facts about Eisenstein integers are standard; we refer to any standard number theory text, such as \cite{IreRos90}, for details.

Let $\mg=\frac{-1}{2}+\frac{\sqrt{3}}{2}i\in\CC$, so that $\mg^2+\mg+1=0$ and $\mg$ is a primitive third root of unity. Then $\Zw$ is the ring of Eisenstein integers, and $\Zw\cong\mathcal{O}_3$ is the ring of integers in the field $\QQ(\mg)\cong\QQ(\sqrt{-3})$. The group of units in $\Zw$ is $U=\{\pm 1,\pm\mg,\pm\mg^2\}$. 
Note $\overline{\mg}=\mg^2$.

The ring $\Zw$ is a Euclidean domain. A Euclidean valuation is given by the \emph{norm} $N\from\Zw\to\NN\cup\{0\}$, where
\[N(a+b\mg)=|a+b\mg|^2=(a+b\mg)\overline{(a+b\mg)}=(a+b\mg)(a+b\mg^2)=a^2-ab+b^2,\]
for each $a+b\mg\in\Zw$.
This also implies that $\Zw$ is a unique factorisation domain and a principal ideal domain. 

Up to multiplication by a unit, the primes in $\Zw$ are $1-\mg$, the rational primes $p\in\NN$ such that $p\equiv 2\mod{3}$, and $a+b\mg$ and $a+b\mg^2$ such that $a,b \in \ZZ$ satisfy $a^2-ab+b^2=q$ where $q\in\NN$ is a rational prime and $q\equiv 1\mod{3}$. Note that $3$ is not prime in $\Zw$ since $3=(2+\mg)(1-\mg)=-\mg^2(1-\mg)^2$.

Two Eisenstein integers are relatively prime if their only common divisors are units in $\Zw$. Since $\Zw$ is a Euclidean domain, there is a Euclidean algorithm to compute the greatest common divisor of two Eisenstein integers.

The group $\PSLw$ is isomorphic to the Bianchi group $\mathrm{PSL}_2(\mathcal{O}_3)$. Here\\ $\PSLw\cong\SLw/\{\pm I\}$.


\section{Figure Eight Knot Complement}
\label{Sec:fig8}

Let $K\subset S^3$ be the figure eight knot.  We work in the upper half-space model of hyperbolic $3$-space  which we refer to as $\HH^3$. Then $\bdy\HH^3=\hat{\CC}=\CC\cup\{\infty\}$. It is well known that $S^3 \setminus K$ has a complete hyperbolic structure, unique up to isometry; see~\cite{pur20}.

The fundamental group of $S^3\setminus K$ is given by the Wirtinger presentation coming from a reduced alternating diagram as
\[\pi_1(S^3\setminus K)=\langle x,y\mid y=wxw^{-1}\rangle,\]
where $w=x^{-1}yxy^{-1}$. Riley~\cite{ril75} showed that there is a discrete faithful representation $\rho\from\pi_1(S^3\setminus K)\to\SLC$ where
\begin{equation}
\label{Eqn:AB_defn}
A=\rho(x)=\mqty(1&1\\0&1),
\quad
 B=\rho(y)=\mqty(1&0\\-\mg&1).
\end{equation}
Letting $\pi\from\SLC\to\PSLC$ be the map that sends $W$ to $\pm W$, then $\pi\circ\rho$ is a discrete faithful $\PSLC$-representation of $\pi_1(S^3\setminus K)$.
(This representation coincides with that of~\cite[p60]{mr03}, but is different from ~\cite[Equation~5.1]{pur20}.)

Let $\Gm=\rho(\pi_1(S^3\setminus K))$. Then
\[\Gm=\langle A,B\mid B=CAC^{-1}\rangle,\]
where $C=\rho(w) = A^{-1}BAB^{-1}=\smqty(0&\mg\\1+\mg&1-\mg)$. Note that $\Gm$ is a subgroup of $\SLw$.

By an abuse of notation, we will also denote $\pi(\Gm)$ by $\Gm$ since $\pi(\Gm)\cong\Gm$.
Then $\Gm$ is the holonomy group of $S^3\setminus K$, so $S^3\setminus K\cong \HH^3/\Gm$. Note that $\Gm$ is an index 12 subgroup of $\PSLw$; see~\cite[Rmk 2]{hat83} or~\cite{ril75}. 

A fundamental domain for the action of $\Gm$ on upper half-space is given by the union of two regular ideal tetrahedra with vertices on $\hat{\CC}$ at $0,1,-\mg^2,\infty$, and $0,\mg,-\mg^2,\infty$ respectively. (Note $-\omega^2 = \frac{1}{2}+\frac{\sqrt{3}}{2} i = e^{i\pi/3}$.) The face-pairing isometries on the fundamental domain for $\Gm$ are given by
\[
A=\mqty(1&1\\0&1), \quad
D=\mqty(1&-1\\-\mg&1+\mg), \quad
E=\mqty(1-\mg&-1\\\mg^2&1+\mg),
\]
where $D = BA^{-1}$ is loxodromic, and $E = BA^{-1}B^{-1}$ is parabolic fixing $-\mg^2$. Indeed, one can check that $A$ sends $(0,\omega,\infty) \mapsto (1,-\omega^2,\infty)$, $D$ sends $(0,1,\infty) \mapsto (\omega,0,-\omega^2)$ and $E$ sends $(0,1,-\omega^2) \mapsto (\omega,\infty,-\omega^2)$.

\begin{figure}[h]
	\centering
	\includegraphics[scale=1.5]{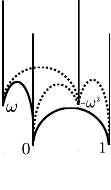}
	\caption{The fundamental domain of the figure eight knot complement consists of two regular ideal tetrahedra.}
	\label{Fig:fig8fundreg}
\end{figure}

This fundamental domain, together with face-pairing maps and holonomy representation, yields a developing map $\mathfrak{D}$ for the complete hyperbolic structure on $S^3 \setminus K$. We refer to this data as the \emph{canonical hyperbolic structure} on $S^3 \setminus K$.

Define $\tilde{C}=AB^{-1}A^{-1}B=\smqty(0&-\mg\\\mg^2&1-\mg)$, and define $L=\tilde{C}^{-1}C=\smqty(-1&2+4\mg\\0&-1)$. Then $L$ is the image under $\rho$ of a longitude of $K$. (Maclachlan-Reid~\cite{mr03} has $L=\smqty(1&-2-4\mg\\0&1)$, which is only true in $\PSLC$. Riley~\cite{ril75} uses $L^{-1}$.)

Let $\Gm_\infty$ be the subgroup of $\Gm$ which stabilises infinity. Then $\Gm_\infty=\langle A,L\rangle\cong\ZZ^2$. Let $\Gm_0$ be the subgroup of $\Gm$ which stabilises $0$. Then $\Gm_0=C\Gm_\infty C^{-1}=\langle B, CLC^{-1}\rangle$.
Note that $A$, $B$, and $L$ are all parabolic elements of $\PSLw$, and have trace $\pm 2$.

The developing image of (the lifts of) a sufficiently small cusp neighbourhood in $S^3 \setminus K$ is an infinite collection of disjoint horoballs in $\HH^3$, centred at the ideal vertices of the fundamental domain, and their translates under the holonomy group $\Gamma$. We expand such a cusp neighbourhood until these horoballs first become tangent, or equivalently, until the cusp neighbourhood first becomes self-tangent. This yields a \emph{maximal cusp}. The developing image of the boundary of this maximal cusp is a collection of horospheres in $\HH^3$, again centred at each ideal vertex of the fundamental domain, and its translates. Among these horospheres, each pair is either tangent or disjoint. We refer to these horospheres in the canonical hyperbolic structure as \emph{canonical horospheres} of $S^3 \setminus K$.

In the fundamental domain, the canonical horospheres appear as (a portion of) a Euclidean plane at height $1$ (centred at $\infty$), and (portions of) Euclidean spheres of diameter $1$ (centred at $0,1,-\omega^2,\omega$). Indeed, the face-pairing transformations $A,D,E$ send the horospheres at height or diameter $1$ at each vertex of each face, to the corresponding height or diameter $1$ horospheres at the vertices of the paired face. (This is most easily seen using spinors for the horospheres as described below in \refsec{spinors}: $A$ sends 
$\smqty(0 \\ 1), \smqty(\omega \\ 1), \smqty (1 \\ 0) \mapsto \smqty (1 \\ 1), \smqty(-\omega^2 \\ 1), \smqty(1 \\ 0)$; $D$ sends $\smqty(0 \\ 1), \smqty ( 1 \\ 1 ), \smqty(1 \\ 0) \mapsto \smqty( -1 \\ -\omega^2), \smqty( 0 \\ 1 ), \smqty(1 \\ -\omega)$; $E$ sends $\smqty(0\\1), \smqty(1\\1), \smqty(-\omega^2\\1) \mapsto \smqty( -1 \\ -\omega^2), \smqty(-\omega\\0), \smqty(-\omega^2\\1)$.)

Translates of these horospheres under $\Gamma_\infty$ then ensure that the canonical horospheres contain a horosphere at each point of $\Zw\subset\CC$. Taking all translates under $\Gamma$, the canonical horospheres consist of one centred at every point of $\QQ(\mg)\cup\{\infty\}$; see~\cite{hat83} or Corollary~\ref{Cor:horosphere_centres} below.


\section{Spinors}
\label{Sec:spinors}

For our purposes, a \emph{spinor} or \emph{spin-vector} is just a pair of complex numbers, not both zero. The third author~\cite{mat23} has shown that there is a bijective correspondence between spinors and spin-decorated horospheres in $\HH^3$. 

As a horosphere $\h$ is isometric to the Euclidean plane, there is a well-defined notion of parallel direction field on $\h$. A \emph{decoration} is a parallel direction field on $\h$, or equivalently, a parallel unit tangent vector field on $\h$. A decoration on $\h$ naturally determines two frame fields on $\h$, given by $(N, V, N \times V)$ where $N$ is one of the two smooth unit normal vector fields to $\h$, $V$ is the parallel unit tangent vector field of the decoration, and $\times$ is the cross product, producing a right-handed orthonormal frame. Choosing $N$ to point into, or out of, $\HH^3$ (i.e. out of the horoball bounded by $\h$, or into it), the decoration determines an \emph{inward} and \emph{outward frame field} $F^{in}, F^{out}$ along $\h$. Each frame field has two lifts to the double cover of the frame bundle of $\HH^3$. An \emph{inward} or \emph{outward} spin-decoration on $\h$ is a continuous lift of $F^{in}$ or $F^{out}$ to the double cover of the frame bundle. Each of the two inward spin-decorations has an associated outward spin-decoration, and vice versa: from an outward spin-decoration $W^{out}$, the associated inward spin-decoration is obtained by rotating $W^{out}$ by $\pi$ about $V$ at each point of $\h$; and from an inward spin-decoration $W^{in}$, the associated outward spin-decoration is obtained by rotating $W^{in}$ by $-\pi$ about $V$ at each point of $\h$. A \emph{spin-decoration} on $\h$ is a pair $(W^{in}, W^{out})$ of associated inward and outward spin-decorations. See \cite{mat23} for further details.

A spin-decoration on $\h$ determines a decoration on $\h$. A decoration on $\h$ naturally lifts to two spin-decorations, related by $2\pi$ rotation. For the purposes of finding distances between horospheres, and for much of this paper, it suffices to consider decorations only. Spin vectors correspond bijectively with spin-decorated horospheres; spin vectors up to sign correspond bijectively with decorated horospheres.

If $\eta\neq0$, then as proved in \cite{mat23}, the spinor $\smqty(\xi\\ \eta)$ corresponds to the horosphere $\h$ centred at $\xi/\eta$  of height $|\eta|^{-2}$ in the upper half space model.
At the point of $\h$ which is highest in the upper half space model, or ``north pole", the tangent plane of $\h$ is parallel to $\CC$ and hence can be identified with $\CC$. We can therefore describe directions on $\h$ by nonzero complex numbers. The spin-decoration on $\h$ projects to the decoration given by a direction field at the north pole of $\h$ in the direction $i/\eta^2$. 

If $\eta = 0$, The spinor $\smqty(\xi \\ \eta) = \smqty(\xi\\0)$ corresponds to the horosphere $\h$ centred at $\infty$ of height $|\xi|^2$. Since this horosphere appears in the upper half space model as parallel to $\CC$, directions along it can be described by nonzero complex numbers. The spin-decoration on $\h$ projects to the decoration given ay a line field in the direction $i\xi^2$. 

This correspondence extends to an $SL_2(\CC)$-equivariant action on both spinors, and spin-decorated horospheres.

If we endow one canonical horosphere of $S^3 \setminus K$ with a spin-decoration, then every other canonical horosphere in the canonical hyperbolic structure on $S^3 \setminus K$ also obtains a spin-decoration, by equivariance under the action of the holonomy group $\Gamma \subset \SLC$.

Recall that the subgroup $\Gamma_\infty$ of $\Gamma$ is that which stabilises infinity, and hence stabilises the canonical horosphere $\h_\infty$ at infinity (which is a Euclidean plane at height $1$ in the upper half space model). The action of $\Gamma_\infty = \langle A,L \rangle$ preserves $\h_\infty$, and preserves any direction on $\h_\infty$, but does \emph{not} preserve a spin decoration on the horosphere. As $A$ has trace $2$, it does preserve spin decorations on $\h_\infty$, e.g. sending $\smqty(1 \\ 0) \mapsto \smqty(1 \\ 0)$. However, $L$ has trace $-2$, so does not preserve spin-decorations, and sends a spin-decoration to the spin-decoration related by a $2\pi$ rotation: indeed, $L$ sends $\smqty (1 \\ 0) \mapsto \smqty(-1 \\ 0)$. However, $\Gamma_\infty$ does stabilise this pair of spin-decorations, which correspond to a decoration on $\h_\infty$.

Thus, if we choose one spin-decoration on one canonical horosphere of $S^3 \setminus K$, then by equivariance under the holonomy group $\Gamma$, each canonical horosphere of $S^3 \setminus K$ naturally obtains \emph{two} spin-decorations, which project to the same decoration.

We choose the spinor $\smqty(1\\0)$, which corresponds to the horosphere $z=1$ centred at $\infty$ with spin-decoration described by $i$. The orbit of $\smqty(1\\0)$ under $\Gamma$ thus corresponds to a pair of spin-decorations, related by $2\pi$ rotation, (and hence a decoration,) on each canonical horosphere of $S^3 \setminus K$. We refer to these spinors as the \emph{canonical spinors} of $S^3 \setminus K$.


\section{Actions of $\Gm$ on $\hat{\CC}$ and $\CC^2$}
\label{Sec:holonomy_action}

Note that $\PSLC$ acts on $\hat{\CC}=\CC\cup\{\infty\}$ by M\"obius transformations, which extends to an action on upper half-space $\HH^3$. It follows that $\PSLw$ and hence $\Gm$ act on $\hat{\CC}$ by M\"obius transformations. Under this action, each element of $\Gm_\infty$ acts as a translation on $\CC$.

The group $\SLC$ acts on $\CC^2$ by matrix multiplication. Hence $\Gm$ also acts on spinors. 

It is important for us to understand the action of $\langle B\rangle$ on $\hat{\CC}$ and $\CC^2$. To do so, we will consider the isometric circles of $B$ and $B^{-1}$.

Let $F=\smqty(a&b\\c&d)\in\SLC$ and suppose $c\neq 0$. Then $F(\infty)\neq\infty$. The isometric circle of $F$, denoted $\Ii(F)$ is the set of points in $\hat{\CC}$ such that $|F'(z)|=1$; see~\cite[p19-21]{mar07}. Since $F'(z)=\frac{1}{(cz+d)^2}$, we have 
\[\Ii(F)=\{z\in\hat{\CC}\mid |cz+d|=1\}=\left\{z\in\CC\mid \left|z+\frac{d}{c}\right|=\frac{1}{|c|}\right\}.\]

Let $\Bb(F)$ be the disk in $\CC$ bounded by $\Ii(F)$. Then $\Bb(F)=\{z\in\CC\mid\left|z+\frac{d}{c}\right|<\frac{1}{|c|}\}$.

\begin{lemma}
Suppose $\h$ is a canonical horosphere of the canonical hyperbolic structure on $S^3\setminus K$ with centre contained in $\Bb(F)\setminus\{\frac{-d}{c}\}$. Then the height of $F(\h)$ is strictly larger than the height of $\h$.
\end{lemma}

\begin{proof}
Suppose that a spinor associated to $\h$ is given by $\xe$. Then $\h$ has centre $\xef$ and height $\frac{1}{|\eta|^2}$. Since $\xef\in\Bb(F)\setminus\{\frac{-d}{c}\}$, we have $0<|\xef+\frac{d}{c}|<\frac{1}{|c|}$. Now $F\xe=\smqty(a\xi+b\eta\\c\xi+d\eta)$, so $F(\h)$ has centre $\frac{a\xi+b\eta}{c\xi+d\eta}\in\CC$ and height $\frac{1}{|c\xi+d\eta|^2}$. Then
\[\frac{1}{|c\xi+d\eta|^2}=\frac{1}{|c\eta|^2|\xef+\frac{d}{c}|^2}>\frac{|c|^2}{|c\eta|^2}=\frac{1}{|\eta|^2}.\]
\end{proof}

Note that, if $\xef=\frac{-d}{c}$, then $F(\frac{-d}{c})=\infty$ and $F$ sends a horosphere centred at $\frac{-d}{c}$ to a horosphere centred at $\infty$. Moreover, if $\xef\in\Ii(F)$, then the height of the horosphere corresponding to $\xe$ is equal to that of $F\xe$.

It is a property of the isometric circle that $F(\Ii(F))=\Ii(F^{-1})$, and that $F(\Bb(F))$ is equal to the exterior of $\Ii(F^{-1})$ in $\hat{\CC}$.

In particular, for the matrix $B$ of \refeqn{AB_defn},
\[\Ii(B)=\{z\in\CC\mid |-\mg z+1|=1\}=\{z\in\CC\mid |z-\mg^2|=1\},\]
and
\[\Ii(B^{-1})=\{z\in\CC\mid |\mg z+1|=1\}=\{z\in\CC\mid |z+\mg^2|=1\}.\]

Let $\Pp$ be the pencil of circles in $\hat{\CC}$ which are mutually tangent at $0$ and whose centre lies on the line $\Ll=\{z\in\hat{\CC}\mid z=\ell\mg^2, \ell\in\hat{\RR}\setminus\{0\}\}$. Let $\Pp^\perp$ be the pencil of circles in $\hat{\CC}$ which are mutually tangent at $0$ and whose centre lies on the line $\Ll^\perp=\{z\in\hat{\CC}\mid z=m(1-\mg), m\in\hat{\RR}\setminus\{0\}\}$. Note that the lines $\Ll$ and $\Ll^\perp$ are orthogonal, and both $\Ii(B)$ and $\Ii(B^{-1})$ belong to $\Pp$. Each circle of $\Pp^\perp$ is orthogonal to every circle of $\Pp$ at their two points of intersection.

If $z\in\Cc$ where $\Cc\in\Pp^\perp$, then $B(z)\in\Cc$. The action of $\langle B\rangle$ on $\hat{\CC}$ can be described as translations along circles of $\Pp^\perp$.

\begin{example}
Let $\h$ be the horosphere of height 1 centred at 1. We will consider its orbit under the action of $\langle B\rangle$. The action by M\"obius transformation gives $B^k(1)=\frac{1}{-k\mg+1}=\frac{k+1+k\mg}{k^2+k+1}$ for $k\in\ZZ$.

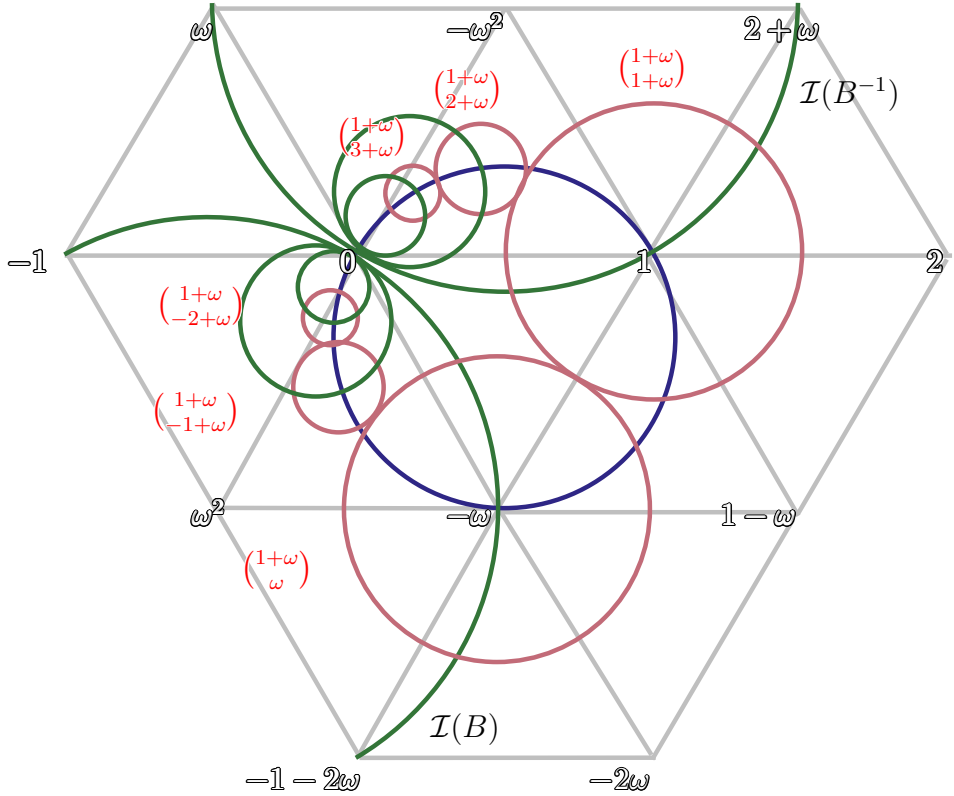
\begin{figure}[h]
\centering
$$\vcenter{\hbox{\begin{overpic}[scale = 2]{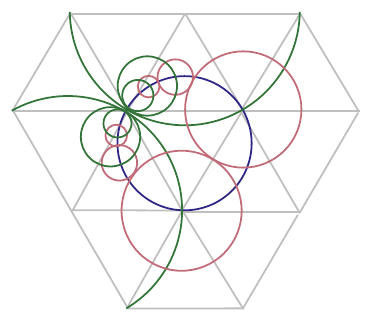}
 \put(290, 260){$\mathcal{I}(B^{-1})$}
\put(150, 20){$\mathcal{I}(B)$}
\put(-10, 196){\contour{black}{\textcolor{white}{$\pmb{-1}$}}}
\put(116, 196){\contour{black}{\textcolor{white}{$\pmb 0$}}}
\put(228, 196){\contour{black}{\textcolor{white}{$\pmb 1$}}}
\put(338, 196){\contour{black}{\textcolor{white}{$\pmb{2}$}}}
\put(60, 284){\contour{black}{\textcolor{white}{$\pmb \omega$}}}
\put(156, 284){\contour{black}{\textcolor{white}{$\pmb{-\omega^2}$}}}
\put(269, 284){\contour{black}{\textcolor{white}{$\pmb{2+\omega}$}}}
\put(60, 100){\contour{black}{\textcolor{white}{$\pmb {\omega^2}$}}}
\put(156, 100){\contour{black}{\textcolor{white}{$\pmb{-\omega}$}}}
\put(260, 100){\contour{black}{\textcolor{white}{$\pmb{1-\omega}$}}}
\put(80, 0){\contour{black}{\textcolor{white}{$\pmb {-1-2\omega}$}}}
\put(210, 0){\contour{black}{\textcolor{white}{$\pmb -2\omega$}}}
\put(220, 270){\contour{white}{\fontsize{12}{12}{\textcolor{red}{$\binom{1+\omega}{1+\omega}$}}}}
\put(150, 262){\contour{white}{\fontsize{12}{12}{\textcolor{red}{$\binom{1+\omega}{2+\omega}$}}}}
\put(114, 244){\contour{white}{\fontsize{12}{12}{\textcolor{red}{$\binom{1+\omega}{3+\omega}$}}}}
\put(78, 80){\contour{white}{\fontsize{12}{12}{\textcolor{red}{$\binom{1+\omega}{\omega}$}}}}
\put(44, 140){\contour{white}{\fontsize{12}{12}{\textcolor{red}{$\binom{1+\omega}{-1+\omega}$}}}}
\put(46, 180){\contour{white}{\fontsize{12}{12}{\textcolor{red}{$\binom{1+\omega}{-2+\omega}$}}}}
\end{overpic} }} $$
\caption{The green circles are $B^k(\Ii(B))$. Each of the green circles belongs to $\Pp$. The blue circle $\Cc$ belongs to $\Pp^\perp$ and is orthogonal to each green circle. The red circles are the horospheres $B^k(\h)$, and they are labelled by their corresponding spinors. Note that the centre of each horosphere lies on the intersection of $\Cc$ with a green circle}
\end{figure}

Let $\Cc$ be the circle of $\Pp^\perp$ which passes through $1$. Then $\Cc=\{z\in\CC\mid |z-(\frac{1-\mg}{3})|=\frac{1}{\sqrt{3}}\}$. Notice that the centre of $B^k(\h)$ lies on $\Cc$ for each $k\in\ZZ$.

In the canonical hyperbolic structure, the spinor corresponding to the horosphere centred at $\infty$ is $\smqty(1\\0)$. Hence a spinor corresponding to the horosphere centred at $0$ is $C\smqty(1\\0)=\smqty(0\\1+\mg)$. The spinor corresponding to the horosphere $\h$ centred at $1$ is $AC\smqty(1\\0)=\smqty(1+\mg\\1+\mg)$. Then the action on spinors gives
\[B^k\mqty(1+\mg\\1+\mg)=\mqty(1&0\\-k\mg&1)\mqty(-\mg^2\\-\mg^2)=\mqty(-\mg^2\\k-\mg^2)=\mqty(1+\mg\\k+1+\mg).\]
The centre of the corresponding horosphere $B^k(\h)$ is at $\frac{-\mg^2}{k-\mg^2}=\frac{k+1+k\mg}{k^2+k+1}$, and the horosphere has height $\frac{1}{k^2+k+1}$. 
Since the centre $1$ of $\h$, lies on $B(\Ii(B)) = \Ii(B^{-1})$, the centre of $B^k(\h)$ lies on $B^{k+1}(I(B))$.
Moreover,
\[B^k(\Ii(B))=\left\{z\in\CC\mid \left|z+\frac{\mg^2}{2k-1}\right|=\frac{1}{|2k-1|}\right\},\]
and the centre of the horosphere $B^k(\h)$ lies at the nonzero intersection of $B^{k+1}(\Ii(B))$ with $\Cc$.

As $B$ acts conformally on $\hat{\CC}$, the directions of the line-fields at the north pole on each horosphere $B^k(\h)$ when projected onto $\CC$ make constant angles with each of the circles from $\Pp$ and $\Pp^\perp$ on which they lie. \hfill\qed
 \end{example}

Recall that $\Re(a+b\mg)=a-\frac{b}{2}$ and $\Im(a+b\mg)=\frac{b\sqrt{3}}{2}$. Let 
\begin{eqnarray*}
\Dd&=&\{z\in\CC\mid 0\leq\Re(z)\leq 1, 0\leq\Im(z)\leq\sqrt{3}\} \\
&&\cup\{z\in\CC\mid -1\leq\Re(z)\leq 0, -\sqrt{3}\leq\Im(z)\leq 0\}.
\end{eqnarray*}

See \reffig{fund_domain}.
Then $\Dd$ is a fundamental domain for the action of $\Gm_\infty$ on $\CC$. 
This may seem an unusual choice of fundamental domain but it has the useful property that 
\[
\Dd\subset \Bb(B)\cup\Ii(B)\cup\Bb(B^{-1})\cup\Ii(B^{-1}).
\]
Note that
\[
\Dd\cap\Ii(B)=\{0,-1,-2-2\mg,-1-2\mg\}
\quad \text{and} \quad
\Dd\cap\Ii(B^{-1})=\{0,1,2+2\mg,1+2\mg\}.
\]


\section{Eisenstein integer horospheres}
\label{Sec:Zw_horospheres}

\begin{lemma}
Let $\xe$ be in the orbit of $\smqty(1\\0)$ under the action of $\Gm$. Then $\xi$ and $\eta$ are relatively prime Eisenstein integers.
\end{lemma}

\begin{proof}
By assumption, there exists $W\in\Gm$ such that $W\smqty(1\\0)=\smqty(\xi\\ \eta)$. Since $W\in\SLw$, it is clear that $\xi,\eta\in\Zw$. Since $\Gm$ is generated by $A$ and $B$, the matrix $W$ can be written as a product of $A$, $B$, and their inverses. 

We have
\[A\mqty(\xi\\ \eta)=\mqty(\xi+\eta\\ \eta), A^{-1}\mqty(\xi\\ \eta)=\mqty(\xi-\eta\\ \eta), B\mqty(\xi\\ \eta)=\mqty(\xi\\-\mg\xi+\eta), B^{-1}\mqty(\xi\\ \eta)=\mqty(\xi\\ \mg\xi+\eta).\]
Then $\dt$ is a common divisor of $\xi$ and $\eta$ if and only if $\dt$ is a common divisor of $\xi\pm\eta$ and $\eta$. Similarly, $\dt$ is a common divisor of $\xi$ and $\eta$ if and only if $\dt$ is a common divisor of $\xi$ and $\pm\mg\xi+\eta$.

Since $1$ and $0$ are relatively prime, it follows that $\xi$ and $\eta$ are relatively prime.
\end{proof}

\begin{lemma}
Let $\xi$ and $\eta$ be relatively prime Eisenstein integers. Then there exists $W\in\Gm$ such that $W\smqty(1\\0)=u\smqty(\xi\\ \eta)$ for some unit $u\in U$.
\end{lemma}

The idea of the proof is to apply $B$ or $B^{-1}$ whenever possible to increase the magnitude of the second entry of the corresponding spinor, increasing the height of the associated horosphere, as seen in the upper half space model. If not possible, apply an element of $\Gm_\infty$ to move to somewhere that it is possible. Eventually the second entry of the associated spinor decreases to $0$, and the associated horosphere is centred at $\infty$, and we arrive at our original spinor $\smqty(1\\0)$.

\begin{proof}
Define a function $f\from\Zw^2\to\NN\cup\{0\}$ by $f(\kp)=|\eta|^2$ where $\kp=\xe$ is a spinor. If $\h$ is the spin-decorated horosphere corresponding to $\kp$, then $f(\kp)$ is the reciprocal of the height of $\h$ whenever $\eta\neq 0$, and $f(\kp)=0$ when $\eta=0$. Note that $f(\kp)$ is just the norm of $\eta$: if $\eta=a+b\mg$, then $|\eta|^2=a^2-ab+b^2\geq 0$. 

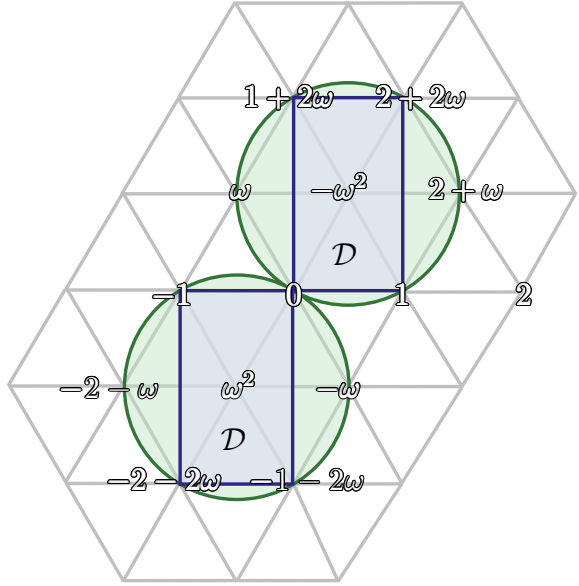
\begin{figure}[h]
$$\vcenter{\hbox{\begin{overpic}[scale = 1.5]{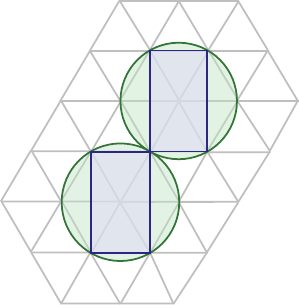}
\put(81, 50){$\mathcal{D}$}
\put(123, 120){$\mathcal{D}$}
\put(55, 105){\contour{black}{\textcolor{white}{$\pmb{-1}$}}}
\put(106, 105){\contour{black}{\textcolor{white}{$\pmb 0$}}}
\put(147, 105){\contour{black}{\textcolor{white}{$\pmb 1$}}}
\put(193, 105){\contour{black}{\textcolor{white}{$\pmb{2}$}}}
\put(85, 145){\contour{black}{\textcolor{white}{$\pmb \omega$}}}
\put(115, 145){\contour{black}{\textcolor{white}{$\pmb {-\omega^2}$}}}
\put(160, 145){\contour{black}{\textcolor{white}{$\pmb {2+\omega}$}}}
\put(90, 180){\contour{black}{\textcolor{white}{$\pmb {1+2\omega}$}}}
\put(140, 180){\contour{black}{\textcolor{white}{$\pmb {2+2\omega}$}}}
\put(38, 35){\contour{black}{\textcolor{white}{$\pmb{-2-2\omega}$}}}
\put(20, 70){\contour{black}{\textcolor{white}{$\pmb -2-\omega$}}}
\put(82, 70){\contour{black}{\textcolor{white}{$\pmb{\omega^2}$}}}
\put(117, 70){\contour{black}{\textcolor{white}{$\pmb{-\omega}$}}}
\put(92, 35){\contour{black}{\textcolor{white}{$\pmb -1-2\omega$}}}
\end{overpic} }} $$
\centering
\caption{The isometric circles of $B$ and $B^{-1}$ and their relationship to $\Dd$.}
\label{Fig:fund_domain}
\end{figure}

Let $\kp_i=\smqty(\xi_i\\ \eta_i)$. Set $\kp_0=\xe$. 

If $\frac{\xi_i}{\eta_i}\in\Bb(B)$, let $\kp_{i+1}=B(\kp_i)$. Then, by lemma, $f(\kp_{i+1})=f(B(\kp_i))<f(\kp_i)$. Set $W_i=B$.

If $\frac{\xi_i}{\eta_i}\in\Bb(B^{-1})$, let $\kp_{i+1}=B^{-1}(\kp_i)$. Then, by lemma, $f(\kp_{i+1})=f(B^{-1}(\kp_i))<f(\kp_i)$. Set $W_i=B^{-1}$.

If $\frac{\xi_i}{\eta_i}\notin\Bb(B)\cup\Bb(B^{-1})$, $\frac{\xi_i}{\eta_i}\notin\Dd$ and $\frac{\xi_i}{\eta_i}\neq\infty$, let $\kp_{i+1}=V_i(\kp_i)$, where $V_i\in\Gm_\infty$ such that $V_i(\frac{\xi_i}{\eta_i})\in\Dd$. Then $V_i=A^kL^m$ for some $k,m\in\ZZ$. Recall that $L$ is a product of $A$, $B$ and their inverses. Note that $f(\kp_{i+1})=f(V_i(\kp_i))=f(\kp_i)$, since $V_i=\pm\smqty(1&\af\\0&1)$ for some $\af\in\Zw$. Set $W_i=V_i$.

Continue applying $B$, $B^{-1}$, and $\Gm_\infty$ is this way until either $\frac{\xi_i}{\eta_i}=\infty$, or $\frac{\xi_i}{\eta_i}\in\Dd\cap(\Ii(B)\cup\Ii(B^{-1}))$. Note that at every stage $f(\kp_{i+1})=f(W_i(\kp_i))\leq f(\kp_i)$.

In the latter case, there are seven possibilities for $\frac{\xi_i}{\eta_i}$, and the necessary transformations are given by the following:
\begin{align*}
&&
\infty&=BAB^{-1}(1+2\mg)&
\infty&=BAB^{-1}A^{-1}(2+2\mg)\\
\infty&=B^{-1}AB(-1)&
\infty&=B^{-1}ABA^{-1}(0)&
\infty&=BA^{-1}B^{-1}(1)\\
\infty&=B^{-1}A^{-1}BA(-2-2\mg)&
\infty&=B^{-1}A^{-1}B(-1-2\mg)&
&
\end{align*}
In each of these cases, the appropriate $W_i$ is a product of $A$, $B$ and their inverses, and takes a point of $\Dd\cap(\Ii(B)\cup\Ii(B^{-1}))$ to $\infty$.

Eventually $\frac{\xi_n}{\eta_n}=\infty$, so $\eta_n=0$ and $f(\kp_n)=0$. Since $\xi$ and $\eta$ are relatively prime, it follows by previous lemma that $\xi_n$ and $\eta_n$ are relatively prime. Thus $\xi_n=u$ for some $u\in U$. Let $W^{-1}=W_{n-1}W_{n-2}\ldots W_1W_0$. Then $W^{-1}\xe=\smqty(u\\0)$, so setting $W=(W^{-1})^{-1}$, we have $W\smqty(1\\0)=u^{-1}\xe$ as required.
\end{proof}

Theorem 1.1 is a restatement of the above two lemmas. The following is known~\cite{hat83} but we record it anyway.

\begin{corollary}
\label{Cor:horosphere_centres}
In the canonical hyperbolic structure on $S^3\setminus K$, there is a horosphere centred at every point of $\QQ(\mg)\cup\{\infty\}\subset\hat{\CC}$.
\end{corollary}

\begin{proof}
Let $\alpha\in\QQ(\mg)$. Then there exist $\frac{c}{d},\frac{e}{f}\in\QQ$ such that $\alpha=\frac{c}{d}+\frac{e}{f}\mg=\frac{cf+de\mg}{df}$. Here both $cf+de\mg$ and $df$ belong to $\Zw$. Let $\delta$ be a greatest common divisor of $cf+de\mg$ and $df$ in $\Zw$. Let $\xi=\frac{cf+de\mg}{\delta}$ and let $\eta=\frac{df}{\delta}$. Then $\xi$ and $\eta$ are relatively prime and $\alpha=\frac{\xi}{\eta}$. 

Thus by the theorem, there exists $W\in\Gm$ and $u\in U$ such that $W\smqty(1\\0)=u\smqty(\xi\\ \eta)$. The centre of the horosphere corresponding to the spinor $\smqty(u\xi\\u\eta)$ is at $\alpha$.
\end{proof}


\section{Lambda lengths}
\label{Sec:lambda_lengths}

Let $\kp=\xe$ and $\kp'=\smqty(\xi'\\ \eta')$ be nonzero spinors where $\kp$ is not a multiple of $\kp'$, nor $\kp'$ a multiple of $\kp$, and let $\h$ and $\h'$ be the corresponding spin-decorated horospheres. There is a geodesic $\gm$ in $\HH^3$ with endpoints at the centres of  $\h$ and $\h'$. Suppose $\gm$ is oriented from $\frac{\xi}{\eta}$ to $\frac{\xi'}{\eta'}$. Then $\gm$ intersects $\h$ and $\h'$ orthogonally at the points $p$ and $p'$ respectively. There is an inward frame $F$ at $p$, with first component given by the tangent to $\gm$ at $p$, the second by the decoration on $\h$ at $p$, and the third by their cross-product in $\HH^3$. Similarly, there is an outward frame $F'$ at $p'$. Let $\rho$ be the distance along $\gm$ between $p$ and $p'$. If we parallel transport the frame $F$ along $\gm$ to $p'$, then the first vectors of the frames will agree, but there will be an angle $\theta$ between the vectors describing the decorations. This angle $\theta$ is well defined modulo $2\pi$. If the frames are lifted to the double cover and describe the spin decorations given by $\kp$ and $\kp'$, then $\theta$ is well defined modulo $4\pi$. Define the complex distance between $\h$ and $\h'$ to be $d=\rho+i\theta$. Then we define the complex $\md$-length from $\h$ to $\h'$ to be $\md=e^{d/2}$. See \reffig{3_of_mat23} and ~\cite{mat23} for further details.

\begin{figure}[h]
\def\svgwidth{0.38\columnwidth}
\begin{center}
\begingroup%
  \makeatletter%
  \providecommand\color[2][]{%
    \errmessage{(Inkscape) Color is used for the text in Inkscape, but the package 'color.sty' is not loaded}%
    \renewcommand\color[2][]{}%
  }%
  \providecommand\transparent[1]{%
    \errmessage{(Inkscape) Transparency is used (non-zero) for the text in Inkscape, but the package 'transparent.sty' is not loaded}%
    \renewcommand\transparent[1]{}%
  }%
  \providecommand\rotatebox[2]{#2}%
  \newcommand*\fsize{\dimexpr\f@size pt\relax}%
  \newcommand*\lineheight[1]{\fontsize{\fsize}{#1\fsize}\selectfont}%
  \ifx\svgwidth\undefined%
    \setlength{\unitlength}{247.70801514bp}%
    \ifx\svgscale\undefined%
      \relax%
    \else%
      \setlength{\unitlength}{\unitlength * \real{\svgscale}}%
    \fi%
  \else%
    \setlength{\unitlength}{\svgwidth}%
  \fi%
  \global\let\svgwidth\undefined%
  \global\let\svgscale\undefined%
  \makeatother%
  \begin{picture}(1,0.64212724)%
    \lineheight{1}%
    \setlength\tabcolsep{0pt}%
    \put(0,0){\includegraphics[width=\unitlength,page=1]{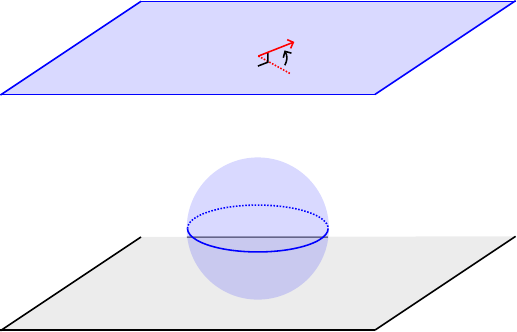}}%
    \put(0.77703916,0.58512719){\color[rgb]{0,0,0}\makebox(0,0)[lt]{\lineheight{1.25}\smash{\begin{tabular}[t]{l}$\mathfrak{h}'$\end{tabular}}}}%
    \put(0,0){\includegraphics[width=\unitlength,page=2]{complex_lambda_lengths.pdf}}%
    \put(0.45375586,0.10699279){\color[rgb]{0,0,0}\makebox(0,0)[lt]{\lineheight{1.25}\smash{\begin{tabular}[t]{l}$\mathfrak{h}$\end{tabular}}}}%
    \put(0,0){\includegraphics[width=\unitlength,page=3]{complex_lambda_lengths.pdf}}%
    \put(0.39200427,0.4027793){\color[rgb]{0,0,0}\makebox(0,0)[lt]{\lineheight{1.25}\smash{\begin{tabular}[t]{l}$\rho$\end{tabular}}}}%
    \put(0.57579002,0.51343701){\color[rgb]{0,0,0}\makebox(0,0)[lt]{\lineheight{1.25}\smash{\begin{tabular}[t]{l}$\theta$\end{tabular}}}}%
    \put(0,0){\includegraphics[width=\unitlength,page=4]{complex_lambda_lengths.pdf}}%
  \end{picture}%
\endgroup%

\caption{Complex distance between horospheres. This is figure 3 of \cite{mat23}.}
\label{Fig:3_of_mat23}
\end{center}
\end{figure}

There is a complex-valued antisymmetric bilinear form $\{-,-\}\from\CC^2\times\CC^2\to\CC$ on spinors defined by the determinant
\[\{\kp,\kp'\}=\mqty|\xi&\xi'\\ \eta&\eta'|.\]
The third author~\cite{mat23} has shown that $\{\kp,\kp'\}=e^{d/2}=\md$.

\begin{proof}[Proof of Theorem~\ref{Thm:lambda_lengths_horospheres}]
Let $\h$ and $\h'$ be spin-decorated horospheres in the canonical hyperbolic structure on $S^3\setminus K$, with corresponding spinors $\kp$ and $\kp'$. Then $\kp,\kp'\in\Zw^2$, so $\md=\{\kp,\kp'\}\in\Zw$.

Let $\alpha\in\Zw\setminus\{0\}$. Then by Theorem~\ref{Thm:Zw_horospheres}, there exists $W\in\Gm$ such that $W\smqty(1\\0)=u\smqty(1\\ \af)$ for some $u\in U$. Let $\h$ be the spin-decorated horosphere corresponding to $\smqty(1\\0)$ and let $\h'$ be the spin-decorated horosphere associated to $\smqty(u\\u\af)$. Then $\md=\smqty|1&u\\0&u\af|=u\af$.  
\end{proof}

\begin{proof}[Proof of Theorem~\ref{Thm:horosphere_distances}]
By Theorem~\ref{Thm:lambda_lengths_horospheres} there exist spin-decorated horospheres in the canonical hyperbolic structure on $S^3\setminus K$ such that $\md=u\af$ for some $u\in U$. Then the complex distance between them is
\[d=2\log\md=2\log(u\af)=2\log(u)+2\log(\af)=2\log(e^{k\pi i/3})+2\log(\af)=\frac{2\ell\pi i}{3}+ 2\log(\af),\]
for some $k,\ell\in\ZZ$.
Then the hyperbolic distance is 
\[\rho=\Re(d)=2\Re(\log(\af))=2\log|\af|.\]

Alternatively, let $\h$ and $\h'$ be horospheres with distinct centres. We may assume without loss of generality that $\h$ is centred at $\infty$ with corresponding spinor $\smqty(u\\0)$ for some $u\in U$, and $\h'$ is centred at $\xef$ with corresponding spinor $\xe$. Then $\h'$ has height $\frac{1}{|\eta|^2}$. Then the hyperbolic distance between $\h$ and $\h'$ is
\[\int_{\frac{1}{|\eta|^2}}^{1} \frac{1}{t}\dd{t}=-\log\frac{1}{|\eta|^2}=2\log|\eta|.\]
Choosing $\h'$ such that $\xe=\smqty(1\\ \af)$ gives a hyperbolic distance of $2\log|\af|$.
\end{proof}


\section{Further questions}
\label{Sec:questions}

We note several directions, which are the subject of further ongoing research.

Firstly, the figure-eight sister should be very similar, and we expect a similar result holds for its complete hyperbolic structure.

Secondly, consider the Whitehead link or Borromean rings. In this case $\Gm$ will be a finite index subgroup of the Picard group $\mathrm{PSL}_2(\ZZ[i])$. There are now infinitely many possible maximal cusp neighbourhoods, but an appropriate choice should still yield some results with connections to the Gaussian integers $\ZZ[i]$.

More generally, one may consider the links discussed by Hatcher in~\cite{hat83}. This will require consideration of the Bianchi groups $\mathrm{PSL}_2(\Oo_d)$ where $\Oo_d$ is the ring of integers in the field $\QQ(\sqrt{-d})$, and $d\in\{1,2,3,7,11\}$. Note that these $\Oo_d$ are Euclidean domains.

Finally, one may consider other $2$-bridge knot complements. These also have a holonomy group $\Gm$ generated by two parabolics $A$, $B$. Does a fundamental domain for the action of $\Gm_\infty$ on $\CC$ lie inside the isometric circles for $B$ and $B^{-1}$? Even if not, there are other parabolic elements in this group, which might be used to cover a fundamental domain of $\Gm_\infty$. However, these are not arithmetic hyperbolic manifolds, and the associated trace fields are no longer quadratic.

\bibliography{fig8lambdalengths}
\bibliographystyle{plain}

\end{document}